\documentclass[11pt]{amsart}
\usepackage{fullpage,verbatim,amssymb}
\usepackage[hidelinks]{hyperref}
\usepackage[dvipsnames]{xcolor}
\usepackage[active]{srcltx}
\usepackage{mathtools}
\mathtoolsset{showonlyrefs}
\usepackage{bbm}
\usepackage{xurl}

\definecolor{blue}{rgb}{0,0,1}
\definecolor{red}{rgb}{1,0,0}
\definecolor{green}{rgb}{0,.6,.2}
\definecolor{purple}{rgb}{1,0,1}

\long\def\red#1\endred{\textcolor{red}{#1}}
\long\def\blue#1\endblue{\textcolor{blue}{#1}}
\long\def\purple#1\endpurple{\textcolor{purple}{ #1}}
\long\def\green#1\endgreen{\textcolor{green}{#1}}
\long\def\teal#1\endteal{\textcolor{teal}{#1}}

\newcommand{\C}{\mathbb{C}}
\newcommand{\R}{\mathbb{R}}

\newcommand{\Z}{\mathbb{Z}}

\newcommand{\legendre}[2]{\ensuremath{\left( \frac{#1}{#2} \right) }}
\newcommand{\pdiv}{\mid\!\mid}

\newcommand{\sm}{\left(\begin{smallmatrix}}
\newcommand{\esm}{\end{smallmatrix}\right)}
\newcommand{\bpm}{\begin{pmatrix}}
\newcommand{\ebpm}{\end{pmatrix}}

\DeclareMathOperator{\new}{new}
\DeclareMathOperator{\ord}{ord}
\DeclareMathOperator{\re}{Re}
\DeclareMathOperator{\tr}{tr}

\DeclarePairedDelimiter\abs{\lvert}{\rvert}%

\newtheorem{theorem}{Theorem}
\newtheorem{lemma}[theorem]{Lemma}
\newtheorem{proposition}[theorem]{Proposition}
\newtheorem{corollary}[theorem]{Corollary}

\theoremstyle{remark}

\newtheorem{remark}[theorem]{Remark}

\numberwithin{theorem}{section}
\numberwithin{equation}{section}

\title{Murmurations in the depth aspect}
\author{Claire Burrin} 
\email{claire.burrin@math.uzh.ch}
\author{Vivian Kuperberg}
\email{vivian.kuperberg@math.ethz.ch}
\author{Min Lee}
\email{min.lee@bristol.ac.uk}
\author{Catinca Mujdei}
\email{catinca.mujdei.23@ucl.ac.uk}
\author{Hsin-Yi Yang}
\email{h.y.yang@uva.nl}

\begin{document}
\maketitle

\begin{abstract}
We compute the murmuration density function for the family of Hecke forms of weight $k$ and prime power level $N=\ell^a$, with $\ell$ a fixed odd prime and  $a\to\infty$ through odd values.
\end{abstract}
\section{Introduction}

In recent years, a new phenomenon known as \emph{murmurations} has emerged in the study of families of $L$-functions. It refers to the appearance of remarkable oscillatory patterns when averaging suitably normalized arithmetic data, such as the correlation of Frobenius traces $a_E(p)$ of elliptic curves or Hecke eigenvalues at primes, with root numbers, over large families. This phenomenon was first uncovered experimentally with the application of machine learning algorithms to datasets of the $L$-Functions and Modular Forms Database (LMFDB) \cite{hlop}; the authors introduced the term \emph{murmurations} by analogy to the murmurations of large flocks of birds, large-scale structured patterns emerging from the many small interactions between individual birds in flight. 


A general conceptual framework for murmurations was proposed by Sarnak \cite{sarnak}, relating the existence and form of murmuration densities to the conductor growth of the family. 
The case of modular forms plays a central role in this theory. For Hecke forms, Zubrilina laid down the main template for the computation of the murmuration density using the Eichler--Selberg trace formula in the (square-free) level aspect \cite{z23}. 
The appearance of oscillatory polynomials in the limiting murmuration densities is then natural from the perspective of these trace computations. This approach has been successfully extended in the weight aspect \cite{bbld} and for Maass forms \cite{BLLDDSHZ}. 
In this paper, we will be concerned with the way in which murmurations interact with arithmetic depth, by considering families of Hecke forms restricted to prime-power levels $N=\ell^a$, with $\ell$ a fixed odd prime and $a\to\infty$. 
After the first version of this paper appeared on the arXiv, Tomczak~\cite{Tom26} extended the depth-aspect picture in several directions, building on the present work: he treated the even-exponent holomorphic case, as well as related Maass and quaternionic families.

\subsection{Determination of the murmuration density}
 Let $H^{\rm new}_k(N)$ be a Hecke basis for trivial character weight $k$ cusp newforms for $\Gamma_0(N)$. For each $f\in H^{\rm new}_k(N)$, $\epsilon_f$ denotes the root number of $f$ and $\lambda _f(n)=a_f(n)n^{(1-k)/2}$ is the $n$th normalized Hecke eigenvalue, having set $a_f(1)=1$. Fix a compact window $E\subset\R_{>0}$ of Lebesgue measure $|E|>0$, and consider the finite density
\begin{align}\label{density}
    \mathcal{M}_{k,\ell,E}(a) := \frac{1}{\#}\sum_{\substack{n/\ell^{a}\in E\\ n \text{ prime}}} \log n \sum_{f\in H^{\rm new}_k(\ell^a)} \epsilon_f \lambda_f(n) \sqrt{n},
\end{align}
where the normalization factor $\#$ stands for the count 
$$
\# := \sum_{\substack{n/\ell^{a}\in E\\ n \text{ prime}}}\log n \sum_{f\in H^{\rm new}_k(\ell^a)}1 \sim
\frac{k-1}{12} \ell^{2a}(1-\ell^{-1})^2 (1+\ell^{-1})|E|;
$$
see \cite[Theorem 1]{martin}. 
Our main theorem (Theorem \ref{thm:main} below) is that the resulting limiting density, as $a\to\infty$, is an explicit integral whose integrand involves a sum over integers arising from the Eichler--Selberg trace formula, with arithmetic weights and a Chebyshev polynomial $U_{k-2}$. Before presenting this result, we need to motivate our choice of definition for  $\mathcal{M}_{k,\ell,E}(a)$ in comparison to the densities studied previously in the level and weight aspects.

The logarithmic weight in \eqref{density} makes for a more convenient handling of the outer sum running over primes $n$, but can be done without. 
The inner weight of $\sqrt{n}$ serves to amplify the correlation between the Hecke eigenvalue $\lambda_f(n)$ and the root number $\epsilon_f$, and is also present in the works discussed above. 
However the finite density considered here differs from that considered earlier in the following aspect. Zubrilina's work \cite{z23} studies murmurations for a single prime $p$, with an average over $N\in [p/y, p/y+p^\delta]$ and $f\in H_k^{\new}(N)$ (and $N$ square-free). 
In this setting, the number of $f$ with conductor $N\leq X$ is $O(X^2)$. 
In the weight aspect \cite{bbld}, the number $f\in H_k(1)$ with conductor $k^2 \leq X$ is $O(X)$, which is of the same order as the conductor. 
Therefore, in line with Sarnak's philosophy \cite{sarnak}, a local averaging (over $p$, with $p/N\in E$) is needed. 
Finally, in our case, if we consider averaging $N=\ell^a\in [p/y, p/y+p^{\delta}]$, at most a few $N$'s are in the given interval; the average has no effect. 
Here too the number of $f$ with the conductor $\ell^a \sim X$ 
is just $O(X)$, and again a local averaging is needed.

Our main result is the following:
\begin{theorem}\label{thm:main}
As $a\to\infty$, we have
\begin{multline}
\frac{\sum_{n/\ell^{2a+1}\in E,\, n\text{ prime }} \log n \sum_{f\in H_k^{\new}(\ell^{2a+1})} \epsilon_f\lambda_f(n)\sqrt{n}}
{\sum_{n/\ell^{2a+1}\in E,\, n\text{ prime }}\log n \sum_{f\in H_k^{\new}(\ell^{2a+1})}1}
\\ = \frac{(-1)^{\frac{k}{2}+1}}{k-1}\frac{2\pi}{(1-\ell^{-1})} 
\frac{1}{|E|}\int_E\sum_{\substack{t\in\mathbb{Z} \\ |t|<2\ell\sqrt{v}}}(\mathbbm{1}_{\{\ell\vert t\}}-\ell^{-1})\bigg(\prod_{\substack{p\nmid \ell t}}\frac{p^2-p-1}{p(p-1)}\bigg)\sqrt{v-\frac{t^2}{4\ell^2}}U_{k-2}\bigg(\frac{t}{2\ell\sqrt{v}}\bigg)\,dv  
\\ + O_{E,k,\ell,C}(a^{-C}),
\end{multline}
where $C>0$ is any positive constant and $U_{k-2}$ is the Chebyshev polynomial given by
\begin{align*}
    U_n(\cos\theta) \coloneqq \frac{\sin((n+1)\theta)}{\sin \theta}.
\end{align*}
\end{theorem}

\begin{remark}
\begin{enumerate} 
\item Note that we consider only odd powers, i.e., $H_k^{\new}(\ell^{2a+1})$. 
In this case, all newforms are twist-minimal forms, meaning that no newforms of level $\ell^{2a+1}$ arise from lower-level newforms twisted by characters. 
Since $H_k^{\new}(\ell^{2a})$ contains non-twist-minimal forms, the corresponding Eichler--Selberg trace formula is slightly messier. 
The even-exponent case is treated in the subsequent work of Tomczak~\cite{Tom26}, where the same limiting density is obtained.

\item We get a stronger error term if we assume the Generalized Riemann Hypothesis (GRH) for Dirichlet $L$-functions, namely 
\[O_{E,k,\ell,\varepsilon}((\ell^{2a+1})^{-\frac{1}{2}+\varepsilon}).\]
This error term is of the same strength as the analogous expression in \cite{bbld}. The key difference between our unconditional and conditional results is an application of the Bombieri--Vinogradov theorem and a Lindel\"of-on-average estimate in place of GRH. 
However, the error term resulting from the Bombieri--Vinogradov application is weaker than the error term coming from GRH, which leads to the different results. We expect that an unconditional analog of \cite[Theorem 1.1]{bbld} should hold with an error term in line with Theorem \ref{thm:main}.
\end{enumerate} 
\end{remark}

\begin{figure}[h]\label{fig:murmurations_depth}
\centering
\includegraphics[scale=.46]{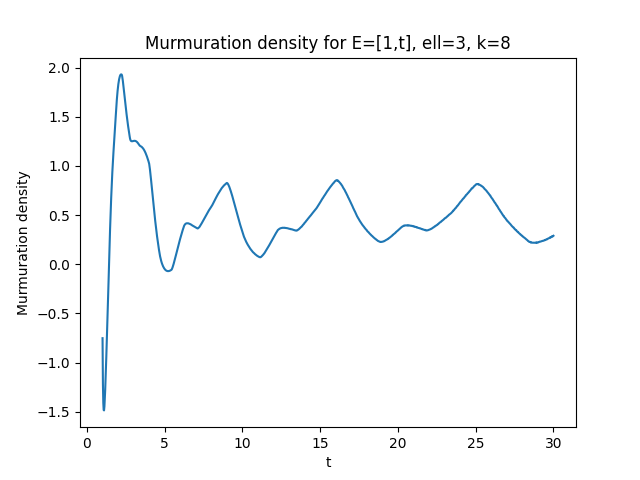}
\includegraphics[scale=.46]{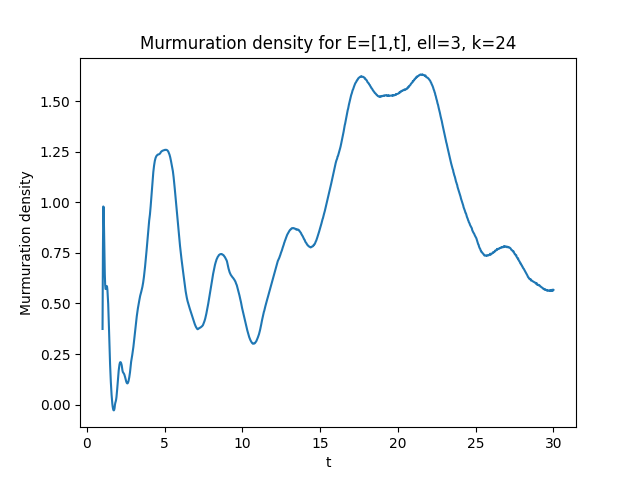}   
\includegraphics[scale=.3]{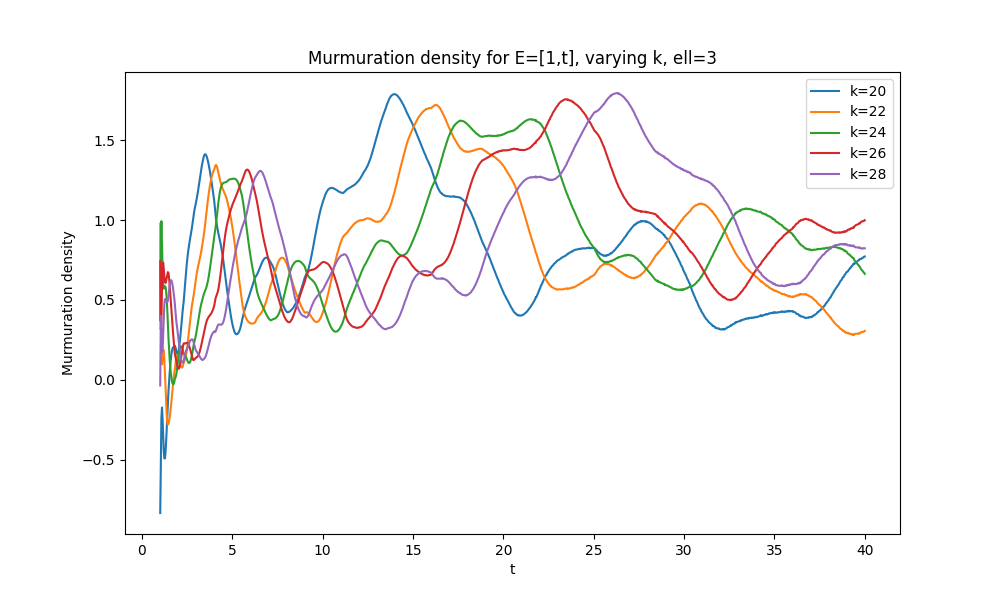}
\includegraphics[scale=.3]{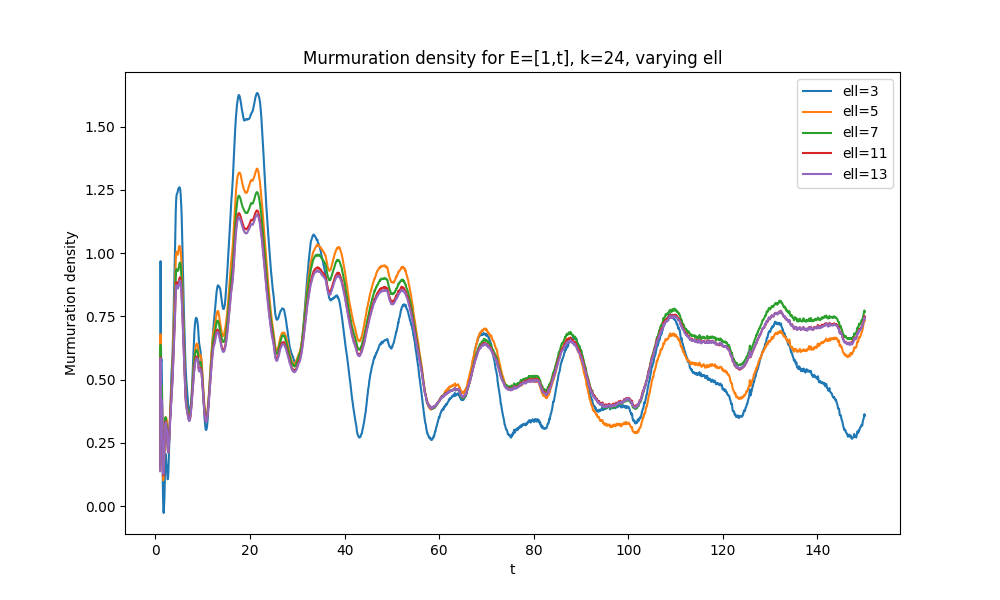}   
\caption{Plots of the asymptotic murmuration density function in Theorem \ref{thm:main} for various values of $\ell,k$ and window $E=[1,t]$}
\end{figure}

The interpretation of the plots in Figure \ref{fig:murmurations_depth} remains elusive. 
The plots suggest a robust oscillatory shape for the windows $E=[1,t]$, but they do not point to an obvious simple limiting behavior as either $k$ or $\ell$ varies. A better understanding of these observed behaviors would be highly desirable.

Naturally we would also like to compare this data to empirical data. A natural strategy would be to compute $\mathcal{M}_{k,\ell,E}(a)$ using newform data from the LMFDB, which can be accessed through Sage. This requires retrieving the Hecke eigenvalues of all elements in a basis $H^{\new}_k(\ell^{a})$. However, for large $a$ these data are not precomputed in the LMFDB, so Sage must instead construct the relevant newform spaces during the computation. Since this involves explicit modular-forms calculations at large level $\ell^a$, the resulting computations become very expensive and introduce substantial computational overhead. As a result, numerical experiments are only feasible for small values of $a$ and do not effectively probe the asymptotic regime we are interested in.

\subsection{Sketch of proof of Theorem \ref{thm:main}}

As in previous works, we start from
\begin{align*}
    \sum_{f\in H_k^{\rm new}(N)} \epsilon_f \lambda_f(n) = (-1)^{k/2}n^{(1-k)/2} \tr(T_n\circ W_N, S_k^{\rm new}(N))
\end{align*}
and apply a variant of the Eichler--Selberg trace formula due to Skoruppa and Zagier \cite{sz}. Specializing to $N=\ell^a$, $\ell\geq3,a\geq 5$ odd, and $n\neq \ell$ prime, we obtain in Section \ref{sec:ESTF} the explicit expression 
\begin{align}
\sum_{f\in H_k^{\rm new}(N)} \epsilon_f \lambda_f(n) = \frac{1}{\pi} \sum_{\substack{t\in \Z\\ t^2-4\ell n<0}} \frac{c_{\ell^{(a+1)/2}}(t)}{\sqrt{\ell}} \cos((k-1) \phi_{t, \ell n}) L(1, \psi_{t^2-4\ell n}),
\end{align}
where $c_q(m)$ is a Ramanujan sum,  $\phi_{t, m} := \arcsin\big(\frac{t}{2\sqrt{m}}\big)\in (-\frac{\pi}{2}, \frac{\pi}{2})$, and $\psi_D(m)$ is a Kronecker symbol.

Exchanging order of summation, we are left with a sum over primes in arithmetic progressions, for which we adapt the approach of \cite{bbld} in Section \ref{sec:proof}. Our main technical result is Lemma \ref{lem:lemma4.5-equivalent-using-bombieri-vinogradov}, which is the analog to \cite[Lemma 4.5]{bbld}. 
Unlike in the latter work, we do not assume the Generalized Riemann Hypothesis. Our unconditional argument relies on an application of the Bombieri--Vinogradov theorem, as well as a Lindel\"of-on-average bound due to \cite{HB}. 
We note that the application of Bombieri--Vinogradov is slightly unusual in that the moduli that appear in our average are all squares, so that the average is taken over a thin set of moduli. Happily, an appropriate variant of the Bombieri--Vinogradov theorem for our purposes was proven in \cite{MR3670199-baker}; we discuss this further in Section \ref{sec:bombieri_vin}.  

\subsection{Acknowledgments} Thanks are due to the organizers of WINE 5 (Women in Numbers Europe) and its sponsors (the Clay Foundation) for putting together such a nice event and for the precious opportunity to connect women in numbers. 
We thank Andrew R. Booker for suggesting the problem of murmurations in the depth aspect during his talk at the conference on Analytic Number Theory and Related Topics at RIMS.
V.K. thanks Yuval Wigderson for helpful conversations concerning Bombieri--Vinogradov. C.M. thanks Alberto Acosta Reche for his suggestion to use a moment estimate in the proof of Lemma \ref{lem:lemma4.5-equivalent-using-bombieri-vinogradov}.
C.B. acknowledges the partial support of SNSF grant 201557 and SNSF-ANR grant 10003145.
M.L. was supported by a Royal Society University Research Fellowship. C.M. was supported by the Engineering and Physical Sciences Research Council [EP/S021590/1], through the EPSRC Centre for Doctoral Training in Geometry and Number Theory (The London School of Geometry and Number Theory) at University College London. H.Y.Y. thanks the Dutch Research Council (NWO) for the support through the grant OCENW.XL21.011.

\section{Eichler--Selberg trace formula}\label{sec:ESTF}

In this section, we fix an odd prime $\ell$, $a\in\Z_{\geq 5}, k\in 2\Z_{\geq 1}$, and a compact window~$E\subset\R_{>0}$ of Lebesgue measure~$|E|>0$.
Let~$S_{k}(\ell^a)$ be the space of holomorphic cusp forms of weight~$k$ and level $\ell^a$, and let~$S_{k}^{\textrm{new}}(\ell^a)$ be the subspace of Hecke newforms in the sense of Atkin--Lehner.
Write~$H_{k}^{\textrm{new}}(\ell^a)$ for a basis of $S_{k}^{\new}(\ell^a)$, with each element normalized to have leading coefficient~$1$.
Set\begin{equation}
    \Sigma \coloneqq \Sigma(\ell, a, k,E) \coloneqq \sum_{\substack{n/\ell^a\in E \\n\text{ prime }}}  \sqrt n\log n 
\sum_{f\in H_k^{\new}(\ell^a)} \epsilon_f \lambda_f(n),
\end{equation} where~$\epsilon_{f}$ denotes the root number of~$f$, and~$\lambda_{f}(n)$ its $n$th Hecke eigenvalue.
In this section, we evaluate the inner sum in the definition of $\Sigma$. First note that 
\begin{equation}\label{eq:trace formula and coeff}
    \sum_{f\in H_k^{\new}(\ell^a)}\epsilon_f\lambda_f(n) = (-1)^{k/2}n^{(1-k)/2}\tr(T_n\circ W_{\ell^a},S_k^{\new}(\ell^a)),
\end{equation}
where $T_n$ and $W_{\ell^a}$ denote the $n$th Hecke operator and $\ell^a$th Atkin--Lehner involution, respectively.
Hence we proceed by applying a relevant trace formula to $\tr(T_n\circ W_{\ell^a},S_k^{\new}(\ell^a))$.

For $D=dL^2$ with $d$ a fundamental discriminant and $L\in \Z_{\geq0}$, we define the Kronecker symbol
\begin{equation}
\psi_D(m) := \legendre{d}{m/(m,L)},
\end{equation}
with the convention~$\psi_0(m)=1$. 
Note that~$\psi_D(\cdot)$ is periodic modulo~$D$, and it is the quadratic character of conductor~$D$ if $D$ is fundamental. We write
\begin{equation}
L(1, \psi_D) := \sum_{m=1}^\infty \frac{\psi_D(m)}{m}.
\end{equation}

\begin{proposition} \label{prop:explicit trace formula}
Let~$a\in\Z_{\geq5}$ be odd. 
For a prime~$n\neq \ell$, we have 
\begin{align}
(-1)^{k/2} n^{(1-k)/2} \tr(T_n\circ W_{\ell^a}, S_k^{\new}(\ell^a)) = \frac{1}{\pi} \sum_{\substack{t\in \Z\\ t^2-4\ell n<0}} \frac{c_{\ell^{(a+1)/2}}(t)}{\sqrt{\ell}} \cos((k-1) \phi_{t, \ell n}) L(1, \psi_{t^2-4\ell n}),
\end{align}
where $c_q(m) := \,\sideset{}{^*}\sum\limits_{x\bmod{q}} e^{2\pi ixm/q}$ is the usual Ramanujan sum, and $\phi_{t, m} := \arcsin\big(\frac{t}{2\sqrt{m}}\big)\in (-\frac{\pi}{2}, \frac{\pi}{2})$.
\end{proposition}
\begin{proof}
    By \cite[p.117]{sz}, the Eichler--Selberg trace formula for $S_{2k-2}^{\new}(N)$ states that
\begin{equation}
\tr(T_n\circ W_{N}, S_{2k-2}^{\new}(N))
= \sum_{M\mid N} \alpha(N/M) s_{k, M}(n, M) \label{eq:tr_2k-2}
\end{equation}
for any $N\in\Z_{\geq 1}$ with $(n,N)=1$, where $\alpha(\cdot)$ is the multiplicative arithmetic function defined by 
\begin{equation}
\alpha(p^j) \coloneqq \begin{cases}
1 & \text{ if } j\in \{0, 3\} \\
-1 & \text{ if } j\in \{1, 2\} \\
0 & \text{ if } j\geq 4
\end{cases}
\end{equation}
for any prime $p$. The factors $s_{k,M}(n,M)$ in \eqref{eq:tr_2k-2} are defined as follows.
Write $Q(n)$ for the greatest integer whose square divides $n$ (i.e. $Q(n)\coloneqq\prod_{p\mid n} p^{[\ord_p(n)/2]}$, where $[\cdot]$ denotes the integral part). Write $\sigma_0(n) := \sum_{d|n} 1$ and $\sigma_1(n) := \sum_{d|n} d$.
Then \cite[\S1, Theorem 1]{sz} asserts that
\begin{align}
s_{k, M}(n, M) &= -\frac{1}{2} \sum_{L\mid M} \sum_{\substack{t\in \Z \\ L\vert t,\, t^2-4Ln\leq 0\\ ((t/L)^2,\,M/L)\text{ $\square$-free}}} p_{2k-2}(t/\sqrt{L}, n) H(t^2-4Ln) \label{eq:skm} \\ 
&\quad\, -\frac{1}{2}\sum_{d\mid n} \min\{d, n/d\}^{2k-3}(Q(M), d+n/d) + \mathbbm{1}_{\{k=2\}} \sigma_0(M) \sigma_1(n),
\end{align}
where~$H(D)$ is the Hurwitz--Kronecker class number of $\abs{D}$ (see \cite[p.120]{sz}), which by \cite[Equation (1.1)]{bl} and Dirichlet's class number formula for~$L(1,\psi_{d})$ is equal to
\begin{equation}\label{eq:H(D)}
    H(D) = \begin{cases} 
        -\frac{1}{12} &\text{ if }D=0 \\
        \frac{\sqrt{|D|}}{\pi} L(1, \psi_D) &\text{ if }D<0
    \end{cases},
\end{equation}
and where we write
\begin{align*}
    p_{k}(t, n) := 
    \begin{cases}
    (k-1)(t/2)^{k-2} &\text{ if }t^2-4n=0 \\
    \frac{\rho_1^{k-1}-\rho_2^{k-1}}{\rho_1-\rho_2} &\text{ if }t^2-4n<0
    \end{cases},
\end{align*}
where $\rho_1$ and $\rho_2$ are the roots of $x^2 - tx + n = 0$. In particular, if $t^2-4Ln<0$, then the roots $\frac{t\pm i\sqrt{4Ln-t^2}}{2\sqrt{L}}$ of $x^2-\frac{t}{\sqrt{L}}x+n=0$ have modulus $\sqrt{n}$ and argument $\pm(\frac{\pi}{2}-\phi_{t,Ln})$. In this case, we compute
\begin{align}
    p_k(t/\sqrt{L},n) = 2(-1)^{\frac{k}{2}+1}L^{\frac{1}{2}}n^{\frac{k-1}{2}}\frac{\cos((k-1)\phi_{t,Ln})}{\sqrt{4Ln-t^2}}. \label{eq:pkn}
\end{align}

We now specify to the case where~$N=\ell^a$. Recalling that $k$ is even and $a \ge 5$ and 
rewriting \eqref{eq:tr_2k-2} for $k$ instead of $2k-2$, we obtain
\begin{align} \label{eq:trl}
    &\tr(T_n\circ W_{\ell^a}, S_{k}^{\new}(\ell^a)) = \sum_{0\leq j\leq a} \alpha(\ell^{a-j}) s_{\frac{k}{2}+1, \,\ell^j}(n, \ell^j) \\ 
    &= s_{\frac{k}{2}+1,\, \ell^a}(n, \ell^a) - s_{\frac{k}{2}+1,\, \ell^{a-1}}(n, \ell^{a-1}) - s_{\frac{k}{2}+1,\, \ell^{a-2}}(n, \ell^{a-2}) + s_{\frac{k}{2}+1, \,\ell^{a-3}}(n, \ell^{a-3}).
\end{align}
 Plugging \eqref{eq:pkn} and \eqref{eq:H(D)} into \eqref{eq:skm}, we obtain
\begin{align}
s_{\frac{k}{2}+1, \,\ell^\alpha}(n, \ell^\alpha) &= \frac{(-1)^{\frac{k}{2}} n^{\frac{k-1}{2}}}{\pi}\sum_{0\leq j\leq\alpha}\ell^{\frac{j}{2}} \sum_{\substack{t\in \Z \\ \ell^{2j} t^2- 4\ell^jn<0 \\ (t^2,\, \ell^{\alpha-j}) \,\square\text{-free}}}\cos((k-1)\phi_{\ell^jt, \ell^jn}) L(1, \psi_{\ell^{2j}t^2-4\ell^jn}) \label{eq:sk/2+1} \\ 
&\quad\,- (\ell^{[\frac{\alpha}{2}]},n+1) + \mathbbm{1}_{\{k=2\}}(\alpha+1) (n+1),
\end{align}
for $\alpha\in\Z_{\geq 0}$.
Note that there cannot be any terms with~$0\leq j\leq\alpha$ and~$t\in\Z$ such that~$\ell^{2j}t^2-4\ell^jn=0$. 
To simplify \eqref{eq:sk/2+1} further, we define
\begin{equation}
A(t, m) := \begin{cases} 
\cos((k-1)\phi_{t, m})L(1, \psi_{t^2-4m}) & \text{ if } t^2-4m<0\\ 0 & \text{ if }t^2-4m=0 
\end{cases},
\end{equation}
which satisfies 
\begin{equation}\label{e:A_jodd}
A(\ell^j t, \ell^j n)
= \ell^{-\frac{j-1}{2}}\sigma_1(\ell^{\frac{j-1}{2}}) A(\ell^{\frac{j+1}{2}} t, \ell n)
\end{equation}
for odd $j\in\Z_{\geq1}$ by \cite[Equation (1.1)]{bl}. We may rewrite \eqref{eq:sk/2+1} as
\begin{align}
s_{\frac{k}{2}+1,\, \ell^\alpha}(n, \ell^\alpha) &= \frac{(-1)^{\frac{k}{2}} n^{\frac{k-1}{2}}}{\pi} \bigg(\ell^{\frac{\alpha}{2}}\sum_{t\in \Z} A(\ell^\alpha t, \ell^\alpha n)+\ell^{\frac{\alpha-1}{2}} \sum_{t\in \Z} A(\ell^{\alpha-1} t, \ell^{\alpha-1} n) + \sum_{0\leq j\leq\alpha-2} \ell^{\frac{j}{2}} \sum_{\substack{t\in \Z\\ \ell\nmid t}} A(\ell^j t, \ell^j n)\bigg) \\ 
&\quad\,- (\ell^{[\frac{\alpha}{2}]}, n+1) + \mathbbm{1}_{\{k=2\}} (\alpha+1) (n+1). \label{eq:sklA}
\end{align}
Returning to $\tr(T_n \circ W_{\ell^a}, S_k^{\new}(\ell^a))$, we end up with
\begin{align}
&\tr(T_n \circ W_{\ell^a}, S_k^{\new}(\ell^a)) \\
&= \frac{(-1)^{\frac{k}{2}}n^{\frac{k-1}{2}}}{\pi} \bigg(\ell^{\frac{a}{2}}\sum_{t\in \Z} A(\ell^{a} t, \ell^a n) + \ell^{\frac{a-2}{2}}\sum_{\substack{t\in \Z\\ \ell\nmid t}} A(\ell^{a-2}t, \ell^{a-2}n) - 2\ell^{\frac{a-2}{2}}\sum_{t\in \Z} A(\ell^{a-2}t, \ell^{a-2}n) \\ 
&\quad\quad\quad\quad\quad\quad\quad + \ell^{\frac{a-4}{2}}\sum_{t\in \Z} A(\ell^{a-4}t, \ell^{a-4}n) - \ell^{\frac{a-4}{2}} \sum_{\substack{t\in \Z\\ \ell\nmid t}} A(\ell^{a-4}t, \ell^{a-4}n)
\bigg) \\
&= \frac{(-1)^{\frac{k}{2}}n^{\frac{k-1}{2}}}{\pi} \ell^{\frac{1}{2}}\bigg(\sigma_1(\ell^{\frac{a-1}{2}})\sum_{t\in \Z}A(\ell^{\frac{a+1}{2}} t, \ell n) + \sigma_1(\ell^{\frac{a-3}{2}}) \sum_{\substack{t\in \Z\\ \ell\nmid t}} A(\ell^{\frac{a-1}{2}}t, \ell n) - 2\sigma_1(\ell^{\frac{a-3}{2}}) \sum_{t\in \Z} A(\ell^{\frac{a-1}{2}}t, \ell n) \\
&\quad\quad\quad\quad\quad\quad\quad\quad + \sigma_1(\ell^{\frac{a-5}{2}}) \sum_{t\in \Z} A(\ell^{\frac{a-3}{2}}t, \ell n) - \sigma_1(\ell^{\frac{a-5}{2}}) \sum_{\substack{t\in \Z\\ \ell\nmid t}} A(\ell^{\frac{a-3}{2}}t, \ell n)\bigg) \\
&= \frac{(-1)^{\frac{k}{2}}n^{\frac{k-1}{2}}}{\pi} \ell^{\frac{1}{2}}\bigg((\sigma_1(\ell^{\frac{a-1}{2}})- \sigma_1(\ell^{\frac{a-3}{2}}))\sum_{t\in \Z} A(\ell^{\frac{a+1}{2}}t, \ell n)- (\sigma_1(\ell^{\frac{a-3}{2}}) - \sigma_1(\ell^{\frac{a-5}{2}}))\sum_{t\in \Z} A(\ell^{\frac{a-1}{2}}t, \ell n)\bigg).
\end{align}
Since
\begin{align*}
    \ell^{\frac{1}{2}}(\sigma_1(\ell^{\frac{a-1}{2}})-\sigma_1(\ell^{\frac{a-3}{2}}))=\ell^{\frac{a}{2}},\quad \ell^{\frac{1}{2}}(\sigma_1(\ell^{\frac{a-3}{2}})-\sigma_1(\ell^{\frac{a-5}{2}}))=\ell^{\frac{a}{2}-1}
\end{align*}
and
\begin{align*}
\frac{c_{\ell^{(a+1)/2}}(t)}{\sqrt{\ell}}
&= \begin{cases} 
\ell^{\frac{a}{2}}(1-\ell^{-1}) 
& \text{ if } \ell^{\frac{a+1}{2}}\mid t \\ 
- \ell^{\frac{a}{2}-1} 
& \text{ if } \ell^{\frac{a-1}{2}}\pdiv t \\
0 & \text{ if } \ell^{\frac{a-1}{2}}\nmid t
\end{cases},
\end{align*} 
we obtain the desired result.
\end{proof}

\section{Proof of Theorem \ref{thm:main}}\label{sec:proof}

We turn to the proof of Theorem \ref{thm:main}, making use of the trace formula from the previous section. Suppose $a\in\Z_{\geq5}$ is odd. Combining \eqref{eq:trace formula and coeff} and Proposition~\ref{prop:explicit trace formula}, we obtain 
\begin{align}
\Sigma &=\frac{1}{\pi} 
\sum_{\substack{n/\ell^a\in E\\ n\text{ prime }}} \sqrt{n}\log n \sum_{\substack{t\in \Z\\ t^2-4\ell n<0}} \frac{c_{\ell^{(a+1)/2}}(t)}{\sqrt{\ell}}\cos((k-1)\phi_{t, \ell n}) L(1, \psi_{t^2-4\ell n})  \\
&=  \frac{1}{\pi\sqrt{\ell}} \sum_{t\in \Z} c_{\ell^{(a+1)/2}}(t)
\sum_{\substack{n\in\ell^a E\cap (\frac{t^2}{4\ell}, +\infty) \\ n\text{ prime}}} \cos((k-1)\phi_{t, \ell n}) 
L(1, \psi_{t^2-4\ell n}) \sqrt{n}\log n.\label{eq:reformulating_Sigma}
\end{align}

We now evaluate the inner sum over primes $n$ in a way similar to \cite[Section~4]{bbld}. That is, we reduce to a sum over primes $n$ in arithmetic progressions. To this end, we define the local averages
\begin{align}
\widetilde{\psi}_t^{(\ell)}(m) 
:= \frac 1{\varphi(m^2)} 
\sum_{\substack{n\bmod{m^2}\\ \gcd(n, m)=1}}
\psi_{t^2 - 4\ell n}(m)\quad\text{ and }\quad \widetilde{\psi}_t(m) := \frac{1}{\varphi(m^2)}
\sum_{\substack{n\bmod{m^2}\\ \gcd(n, m)=1}}\psi_{t^2-4n}(m)
\end{align} 
for $m\in\Z_{\geq1}$, $t\in\Z$. The second of these two averages is the same as $\widetilde{\psi}_t(m)$ defined in \cite[Lemma 4.2]{bbld}, so we can study the first average by relating the two.

\begin{lemma}
For fixed $t\in\Z$, the function~$\widetilde{\psi}^{(\ell)}_t(m)$ is multiplicative in $m$.
\end{lemma}
\begin{proof}
Fix any $m_1,m_2 \in \mathbb Z_{\geq1}$ with $(m_1,m_2) = 1$. For $n \in \mathbb Z$, we write $t^2 - 4\ell n= dL^2$ for a fundamental discriminant $d$ and $L \in \mathbb Z_{\ge 0}$. Define $r_i\coloneqq (m_i,L)$ for $i\in\{1,2\}$. Since $(m_1,m_2) = 1$, we have $(r_1,r_2) = 1$ and $(m_1m_2,L) = r_1r_2$. Therefore
\[\psi_{t^2 - 4\ell n}(m_1m_2) 
= \legendre{d}{\frac{m_1}{r_1}\frac{m_2}{r_2}} 
= \legendre{d}{m_1/r_1}\legendre{d}{m_2/r_2} 
= \psi_{t^2 - 4\ell n}(m_1) \psi_{t^2 - 4\ell n}(m_2),\]
so that
\begin{multline*}
\widetilde{\psi}_t^{(\ell)}(m_1m_2) = \frac 1{\varphi((m_1m_2)^2)} \sum_{\substack{n\bmod{(m_1m_2)^2}\\ \gcd(n, m_1m_2)=1}}
\psi_{t^2 - 4\ell n}(m_1)\psi_{t^2 - 4\ell n}(m_2) 
\\ = \frac 1{\varphi(m_1^2)\varphi(m_2^2)} 
\sum_{\substack{n_1\bmod{m_1^2}\\ \gcd(n_1, m_1)=1}} 
\quad \sum_{\substack{n_2\bmod{m_2^2}\\ \gcd(n_2, m_2)=1}} 
\psi_{t^2 - 4\ell(n_1(m_2\overline{m}_2)^2 + n_2(m_1\overline{m}_1)^2)}(m_1)
\\ \times \psi_{t^2 - 4\ell(n_1(m_2\overline{m}_2)^2 + n_2(m_1\overline{m}_1)^2)}(m_2),
\end{multline*}
where $m_1 \overline{m_1} \equiv 1 \bmod m_2^2$ and $m_2 \overline{m_2} \equiv 1 \bmod m_1^2$. 
By \cite[Lemma 4.1]{bbld}, we have
\[\psi_{t^2 - 4\ell(n_1(m_2\overline{m}_2)^2 + n_2(m_1\overline{m}_1)^2)}(m_i) = \psi_{t^2 - 4\ell n_i}(m_i)
\]
for each $i\in\{1,2\}$, so we get
\[\widetilde{\psi}_t^{(\ell)}(m_1m_2) = \widetilde{\psi}_t^{(\ell)}(m_1)\widetilde{\psi}_t^{(\ell)}(m_2),\]
as desired.
\end{proof}

\begin{lemma}\label{lemma:lemma4.3inBBLLD}
Let~$t \in \Z$ with~$\ell|t$. 
For~$\re s > 1$, define
\begin{equation*}
L(s,\widetilde{\psi}_t^{(\ell)}) := \sum_{m\geq1} \frac{\widetilde{\psi}_t^{(\ell)}(m)}{m^s}\quad\text{ and }\quad
    L(s,\widetilde{\psi}_t) := \sum_{m\geq1} \frac{\widetilde{\psi}_t(m)}{m^s}.
\end{equation*}
Then we have
\begin{equation}\label{eq:relating-L-psi-tilde-t-ell-to-L-psi-tilde-t}
    L(s, \widetilde{\psi}_t^{(\ell)}) = L(s,\widetilde{\psi}_t)(1-\ell^{-2s}) = L(s,\widetilde{\psi}_1)P(s,t)(1-\ell^{-2s}),
\end{equation}
where
\begin{equation*}
    L(s,\widetilde{\psi}_1) = \zeta(2s)\zeta(s+2)\prod_p \begin{cases}
    (1-(1+p^{-1})(1 + p^{-s})p^{-s-1}) &\text{ if } p \neq 2 \\ 
    (1-2^{-s})(1-2^{-s-2}) &\text{ if } p = 2 \end{cases}
\end{equation*}
and
\begin{equation*}
    P(s,t) \coloneqq \prod_{p|t} \begin{cases}\frac{1-p^{-s-2}}{1-(1+p^{-1})(1+p^{-s})p^{-s-1}} &\text{ if } p \neq 2 \\ 
    \vspace{-3mm} \\
    \frac{1 + 2^{-s-1} - 2^{-2s}}{1-2^{-s}} &\text{ if } p = 2 ,\, 4 \mid t \\ 
    \vspace{-3mm} \\
    \frac{1 + 2^{-s-2} - 7 \cdot 2^{-2s-3} - 2^{-3s-2}}{(1-2^{-s})(1-2^{-s-2})} &\text{ if } p = 2,\, 4 \nmid t
    \end{cases}.
\end{equation*}
\end{lemma}

\begin{proof}
It suffices to prove the first equation in \eqref{eq:relating-L-psi-tilde-t-ell-to-L-psi-tilde-t}, as the rest follows immediately from \cite[Lemma 4.3]{bbld}.
If $m = p^b$ for some $b\in\Z_{\geq0}$ and $p \ne \ell$, then 
\[
\widetilde{\psi}_t^{(\ell)}(p^b) = \frac 1{\varphi(p^{2b})} \sum_{\substack{n \bmod p^{2b}\\ p\nmid n}} \psi_{t^2 - 4\ell n}(p^b) 
= \frac 1{\varphi(p^{2b})} 
\sum_{\substack{n \bmod p^{2b}\\ p\nmid n}} \psi_{t^2 - 4n}(p^b) 
= \widetilde{\psi}_t(p^b)\]
by \cite[Lemma 4.1]{bbld}.
On the other hand, recalling that $\ell |t$, we have $\psi_{t^2 - 4\ell n}(\ell^b) = 0$ for all $n\neq\ell$ and $b \in \Z_{\geq1}$, and thus $\widetilde{\psi}_t^{(\ell)}(\ell^b) = 0$ if $b \in\Z_{\geq1}$. 
So we obtain by Lemma \ref{lemma:lemma4.3inBBLLD} that
\[L(s, \widetilde{\psi}_t^{(\ell)}) 
= \prod_p \Big(\sum_{j\geq0} \widetilde{\psi}_t^{(\ell)}(p^j) p^{-js}\Big) 
= L(s, \widetilde{\psi}_t) \Big(\sum_{j\geq0} \widetilde{\psi}_t(\ell^j) \ell^{-js}\Big)^{-1}\quad\text{ for }\re s>1.
\]
It is shown in the proof of \cite[Lemma 4.3]{bbld} that
\[\sum_{j\geq0} \widetilde{\psi}_t(\ell^j) \ell^{-js} = (1-\ell^{-2s})^{-1},\]
which completes the argument.
\end{proof}

The following lemma is an adaptation of \cite[Lemma 4.5]{bbld}. 
Importantly, its argument replaces the reliance on GRH with an application of Baker’s theorem and a Lindel\"of-on-average bound; its proof is given in Section \ref{sec:bombieri_vin}.
\begin{lemma}\label{lem:lemma4.5-equivalent-using-bombieri-vinogradov}
Let $\ell$ be a prime, $t \in \mathbb Z$ with $\ell\mid t$ and $A,B \in \mathbb R$ with $\frac{t^2}{4\ell} < A < B$. 
Let $\Phi \in C^1([A,B])$, and set $M \coloneqq \max_{u \in [A,B]} |\Phi(u)|$ and $V \coloneqq \int_A^B |\Phi'(u)|du$. 
Then
\[\sum_{\substack{n \in (A,B] \\ n \text{ prime}}} \Phi(n)L(1, \psi_{t^2 - 4\ell n}) \log n 
= L(1,\widetilde{\psi}^{(\ell)}_t) \int_{A}^{B} \Phi(u)\, du
+  O_{C}\Big((M+V)\frac{B}{(\log B)^C}\Big)
\]
for any $C >0$. 
\end{lemma}

We can now proceed to the proof of Theorem \ref{thm:main}. Suppose that $E=[\alpha,\beta]$. Applying first Lemma \ref{lem:lemma4.5-equivalent-using-bombieri-vinogradov} (with $[A,B]:=\ell^aE\cap(\frac{t^2}{4\ell},+\infty)$ and $\Phi(n):=\cos((k-1)\phi_{t,\ell n})\sqrt{n}\,$) and then \eqref{eq:relating-L-psi-tilde-t-ell-to-L-psi-tilde-t} to \eqref{eq:reformulating_Sigma} gives
\begin{multline}\label{eq:Sigma_after_lemma2.4}
\Sigma 
= \frac{1}{\pi\sqrt{\ell}} \sum_{t\in \Z}  c_{\ell^{(a+1)/2}}(t) \bigg(L(1, \widetilde{\psi}_t)(1-\ell^{-2}) \int_{\max\{\alpha \ell^a,\, t^2/(4\ell)\}}^{\beta \ell^a}  \cos ((k-1) \phi_{t,\ell u})\sqrt{u}\, du 
\\ + O_{\beta,\ell,C}((1+V)\ell^aa^{-C})\bigg).
\end{multline}
Here we have
\begin{align*}
V &\coloneqq \int_{\max\{\alpha \ell^a,\, t^2/(4\ell)\}}^{\beta \ell^a} 
\abs*{\frac{d}{du} (\cos ((k-1) \phi_{t,\ell u})\sqrt{u})} \, du \\ 
&\le \int_{\max\{\alpha \ell^a,\, t^2/(4\ell)\}}^{\beta \ell^a} \frac{1}{\sqrt{u}}\, du + \sqrt{\beta}\ell^{\frac{a}{2}} \int_{t^2/(4\ell)}^{+\infty} \abs*{\frac{(k-1)t}{4\sqrt{\ell u^3}} \frac 1{\sqrt{1-t^2/(4\ell u)}} \sin((k-1)\arcsin(t/(2\sqrt{\ell u})))}du \\ 
&\ll \sqrt{\beta}\ell^{\frac{a}{2}} (k-1)^2 \int_{t^2/(4\ell)}^{+\infty} \frac{t^2}{\ell u^{2}}\frac 1{\sqrt{1-t^2/(4\ell u)}}du \\ 
&\ll_{\beta,k} \ell^{\frac{a}{2}},
\end{align*}
where in the third step we applied the fact that $|\sin((k-1) \theta)| \le (k-1) |\sin \theta|$.
Changing variables $\ell^av = u$ in the integral in \eqref{eq:Sigma_after_lemma2.4}, we get
\begin{multline*}
\Sigma = \frac{1}{\pi\sqrt{\ell}} \sum_{\substack{t\in \Z \\ t^2 < 4\ell^{a+1}\beta}} c_{\ell^{(a+1)/2}}(t) 
\\ \times \bigg(\ell^{\frac{3a}{2}} L(1,\widetilde{\psi}_t) (1-\ell^{-2})\int_{\max\{\alpha,\, t^2/(4\ell^{a+1})\}}^\beta  \cos((k-1)\phi_{t,\,\ell^{a+1}v})\sqrt{v}dv + O_{\beta,k,\ell,C}(\ell^{\frac{3a}{2}}a^{-C}) \bigg).
\end{multline*}
Since $c_{\ell^{(a+1)/2}}(t)=0$ unless $\ell^{\frac{a-1}{2}}\mid t$, in which case $c_{\ell^{(a+1)/2}}(t) = \ell^{(a-1)/2} c_\ell(t/\ell^{(a-1)/2})$, 
and $L(1, \widetilde{\psi}_t)$ only depends on the primes dividing $t$, we can further rewrite
\[\Sigma = \frac{\ell^{2a-1}}{\pi} \sum_{\substack{t\in \Z\\|t|< 2\ell\sqrt{\beta}}} c_{\ell}(t)L(1,\widetilde{\psi}_{\ell t}) (1-\ell^{-2}) \int_{\max\{\alpha,\, t^2/(4\ell^2)\}}^{\beta}  \cos((k-1)\phi_{t,\ell^2v})\sqrt{v} \, dv + O_{\beta,k,\ell,C} (\ell^{2a}a^{-C}). 
\]
The remaining sum is now a finite sum depending on~$\alpha$, $\beta$ and~$\ell$, but not on~$a$. 
It remains to treat the inner integral.
Since~$k$ is even, we have  
\[\sin\left((k-1)\left(\frac{\pi}{2}-\phi_{t,\ell^2 v}\right)\right)
= (-1)^{\frac{k}{2}+1} \cos((k-1)\phi_{t,\ell^2 v})\]
and 
\[\sin\left(\frac{\pi}{2}-\phi_{t,\ell^2 v}\right)
= \cos(\phi_{t,\ell^2 v}) = \frac{\sqrt{4\ell^2 v-t^2}}{2\ell\sqrt{ v}}
= \sqrt{1-\frac{t^2}{4\ell^2 v}}, \]
so that 
\begin{align*}
\cos((k-1)\phi_{t,\ell^2 v}) &= (-1)^{\frac{k}{2}+1} \sqrt{1-\frac{t^2}{4\ell^2 v}}\frac{\sin((k-1)(\frac{\pi}{2}-\phi_{t,\ell^2 v}))}{\sin(\frac{\pi}{2}-\phi_{t,\ell^2 v})} \\ 
&= (-1)^{\frac{k}{2}+1} \sqrt{1-\frac{t^2}{4\ell^2 v}} U_{k-2}\left(\cos\left(\frac{\pi}{2}-\phi_{t,\ell^2 v}\right)\right) \\
&= (-1)^{\frac{k}{2}+1} \sqrt{1-\frac{t^2}{4\ell^2 v}} U_{k-2}\left(\frac{t}{2\ell\sqrt{v}}\right).
\end{align*}
Therefore, we obtain
\begin{multline}\label{e:final_Sigma}
\Sigma = \frac{(-1)^{\frac{k}{2}+1}\ell^{2a-1}}{\pi} \sum_{\substack{t\in \Z\\|t|< 2\ell\sqrt{\beta}}} c_{\ell}(t)L(1,\widetilde{\psi}_{\ell t})(1-\ell^{-2}) 
\int_{E \cap (t^2/(4\ell^2),\,+\infty)} \sqrt{v-\frac{t^2}{4\ell^2}} U_{k-2}\left(\frac{t}{2\ell\sqrt{v}}\right)\, dv 
\\ + O_{E,k,\ell,C} (\ell^{2a}a^{-C}). 
\end{multline}
As $a\in\Z_{\geq5}$ is assumed to be odd, we can rewrite it as $2a+1$ for some $a\in\Z_{\geq2}$.
Now Theorem \ref{thm:main} follows from the explicit expressions
\begin{align*}
    &c_\ell(t) = \mathbbm{1}_{\{\ell\vert t\}}\ell-1, \\
    &P(1,\ell t) =  \prod_{\substack{p\mid \ell t}} \left(\frac{1-p^{-3}}{1-(1+p^{-1})^2 p^{-2}}\right), \\
    &L(1,\tilde \psi_{\ell t}) = L(1,\tilde \psi_1)P(1,\ell t) = \zeta(2) \prod_{\substack{ p\nmid \ell t}}  \left(\frac{1-(1+p^{-1})^2 p^{-2}}{1-p^{-3}}\right) = \zeta(2) \prod_{\substack{ p\nmid \ell t}}  \left( \frac{p^2-p-1}{p(p-1)}\right), \\
    &\zeta(2)=\frac{\pi^2}{6}.
\end{align*}

\section{Bombieri--Vinogradov and necessary variants}\label{sec:bombieri_vin}

In order to remove the reliance on the Riemann Hypothesis from the analog of \cite[Lemma 4.5]{bbld} to obtain Lemma \ref{lem:lemma4.5-equivalent-using-bombieri-vinogradov}, we require a variant of the Bombieri--Vinogradov theorem. The unusual aspect of our setting is that our moduli are always perfect squares, and squares are sparse among all integers. Baker \cite{MR3670199-baker} proved the following version of Bombieri--Vinogradov for this setting:

\begin{theorem}[\cite{MR3670199-baker}, Theorem 1]\label{thm:baker-thm-1}
Let 
\[E(x,q) := \max_{(a,q) = 1} \left| \sum_{\substack{n \le x \\ n \equiv a \bmod q}} \Lambda(n) - \frac x{\phi(q)}\right|. \]
For $\varepsilon > 0$ and $Q \le x^{1/2 - \varepsilon}$, for all $A > 0$,
\[\sum_{Q < q^2 \le 2Q} E(x,q^2) \ll_{A,\varepsilon} x Q^{-1/2} (\log x)^{-A}.\]
\end{theorem}

Theorem \ref{thm:baker-thm-1} bounds $E(x,q^2)$ instead of $\max_{u \le x} E(u,q^2)$, but in fact the latter version holds as a corollary. 
The derivation of this corollary follows along the lines of \cite[Exercise 20]{tao}.

\begin{corollary}
For $\varepsilon > 0$ and $Q \le x^{1/2 - \varepsilon}$, for all $A > 0$,
\[\sum_{Q < q^2 \le 2Q} \max_{u \le x} E(u,q^2) \ll_{A,\varepsilon} x Q^{-1/2} (\log x)^{-A}.\]
\end{corollary}
\begin{proof}
We will write $y_k := k \frac{x}{(\log x)^{A+10}}$ for all $0 \le k \le (\log x)^{A+10}$. 
We then have
\begin{align*}
\sum_{Q < q^2 \le 2Q}& \max_{u \le x} E(u,q^2) \\ 
&= \sum_{Q < q^2 \le 2Q} \max_{0 \le k \le (\log x)^{A+10}} \left(\max_{y_k \le u \le y_{k+1}} E(u,q^2)\right) \\ 
&\ll \sum_{Q < q^2 \le 2Q} \max_{0 \le k \le (\log x)^{A+10}} \left( E(y_k,q^2) + \max_{(a,q^2) = 1} \sum_{\substack{y_k < n \le y_{k+1} \\ n \equiv a \bmod q^2}} \Lambda(n) + \frac{(y_{k+1}-y_k)}{\phi(q^2)} \right) \\ 
&\ll \sum_{Q < q^2 \le 2Q} \max_{0 \le k \le (\log x)^{A+10}} \left( E(y_k,q^2) + \frac{x}{\phi(q^2)(\log x)^{A+9}}\right) \\ 
&\ll \sum_{0 \le k \le (\log x)^{A+10}} \sum_{Q < q^2 \le 2Q} E(y_k,q^2) + \frac{x}{(\log x)^{A+9}} \sum_{Q < q^2 \le 2Q} \frac 1{\phi(q^2)}.
\end{align*}
We can now bound the first term by applying Theorem \ref{thm:baker-thm-1} and the second term by noting that $\phi(n) > c \frac{n}{\log \log n}$ for some absolute constant $c > 0$, to get that the expression above is
\begin{align*}
&\ll_{A,\varepsilon} \sum_{0 \le k \le (\log x)^{A+10}} y_k Q^{-1/2} (\log y_k)^{-2A-10} + \frac{x}{(\log x)^{A+9}} Q^{1/2} \frac{\log \log Q}{Q} \\ 
&\ll_{A,\varepsilon} \frac{x}{(\log x)^{3A+20}} Q^{-1/2} \sum_{0 \le k \le (\log x)^{A+10}} k + xQ^{-1/2} (\log x)^{-A-8} \\ 
&\ll_{A,\varepsilon} xQ^{-1/2} (\log x)^{-A},
\end{align*}
as desired.
\end{proof}

We turn to the proof of Lemma \ref{lem:lemma4.5-equivalent-using-bombieri-vinogradov} in the next section. 

\subsection{Proof of Lemma \ref{lem:lemma4.5-equivalent-using-bombieri-vinogradov}}

The result holds trivially if $B < 2$, so assume $B \ge 2$ and write $I = \{n \in (A,B]: n \text{ prime}\}$, 
noting that $I$ omits the left endpoint. 
For each $m\in\Z_{\geq0}$, we have
\begin{align}
\sum_{n \in I} \Phi(n) \psi_{t^2 - 4\ell n}(m)  \log n 
&= \sum_{\substack{a \bmod m^2 \\ (a,m) = 1}} \sum_{\substack{n \in I \\ n \equiv a \bmod m^2}} \Phi(n)\psi_{t^2 - 4\ell n}(m)  \log n + \sum_{\substack{n \in I \\ n|m}} \Phi(n) \psi_{t^2 -4\ell n}(m) \log n \\
&= \sum_{\substack{a \bmod m^2 \\ (a,m) = 1}} \psi_{t^2 - 4a\ell}(m) \sum_{\substack{n \in I \\ n \equiv a \bmod m^2}}  \Phi(n) \log n + \sum_{\substack{n \in I \\ n|m}} \Phi(n)\psi_{t^2 -4\ell n}(m) \log n, \label{eq:second_step_lemma4.5}
\end{align}
where in the second step we used \cite[Lemma 4.1]{bbld} for the first term. As for the second term, we observe that
\[\Big|\sum_{\substack{n \in I \\ n|m}} \Phi(n)\psi_{t^2 - 4\ell n}(m) \log n \Big| \le M \sum_{\substack{n\in I \\ n |m}} \log n \le M \log m \le M \varphi(m^2).
\]
    
We now treat the first term. 
For $m\in\Z_{\geq1}$, $a$ modulo $m^2$ with $(a,m)=1$, and $x \ge 2$, we write
\[\theta(x;m^2,a) := \sum_{\substack{p \le x \\ p \equiv a \bmod m^2}} \log p
\]
and $E_\theta(x;m^2,a) := \theta(x;m^2,a) - \frac{x}{\varphi(m^2)}$, $E(x;m^2,a) := \sum_{\substack{n \le x \\ n \equiv a \bmod m^2}} \Lambda(n) - \frac{x}{\varphi(m^2)}$, noting that $E_\theta(x;m^2,a) = E(x;m^2,a) + O(\sqrt x)$. 
Then
\begin{equation}\label{eq:consequence_of_grh}
\sum_{\substack{n \in I \\ n \equiv a \bmod m^2}}\Phi(n) \log n 
= \int_A^B \Phi(u) \,\mathrm d\theta(u;m^2,a) = \frac{1}{\varphi(m^2)}\int_A^B \Phi(u)\,du + \int_A^B \Phi(u) \,\mathrm dE_\theta(u;m^2,a).
\end{equation}
Applying integration by parts, the error term is 
\begin{align*}
\Phi(u) &E_\theta(u;m^2,a)\Big|_A^B - \int_A^B E_\theta(u;m^2,a) \Phi'(u) \mathrm du \\
&= \Phi(u) E(u;m^2,a) \Big|_A^B  - \int_A^B E(u;m^2,a) \Phi'(u) \mathrm du + O((M+V)\sqrt B) \\
&\ll (M+V)(\sqrt B + \max_{u \in (A,B]} |E(u;m^2, a)|).
\end{align*}
Summing over all $a$ modulo $m^2$ with $(a,m)=1$, we get
\begin{align}
\sum_{n \in I} &\Phi(n)\psi_{t^2 - 4\ell n}(m)\log n \\
&= \sum_{\substack{a \bmod m^2 \\ (a,m) = 1}} \psi_{t^2 -4a\ell}(m) \left(\frac{1}{\varphi(m^2)}\int_A^B \Phi(u)\mathrm du + O((M+V)(\sqrt B + \max_{u \in [A,B]} |E(u;m^2,a)|))\right) \\
&\quad\, + O(M\varphi(m^2)) \\ 
&= \tilde{\psi}_t^{(\ell)}(m) \int_A^B \Phi(u) \mathrm du + O((M+V) \varphi(m^2)(\max_{u \in [A,B]} E(u,m^2) + \sqrt B)), \label{eq:before_Lfunction}
\end{align}
where we recall that 
\[E(u,m^2) = \max_{a \bmod m^2}\abs{E(u;m^2,a)}.\]

We now proceed by ``truncating'' the $L$-values $L(1,\psi_{t^2-4\ell n})$. 
To do so, define the smoothed sum
\[S(x,\psi) \coloneqq \sum_{m\geq1}\frac{\psi(m)}{m}e^{-m/x},\]
where $x\gg1$ and $\psi$ is a non-trivial  Dirichlet character. 
For $\sigma\in\R$, let $(\sigma)$ denote the vertical path $\R\rightarrow\C:\tau\mapsto \sigma+i\tau$.
By Mellin inversion, we have for $\sigma\in\R_{>0}$ that
\begin{align}
S(x,\psi) &= \frac{1}{2\pi i}\int_{(\sigma)}L(1+s,\psi)\Gamma(s)x^s\,ds \\
&= L(1,\psi) + \frac{1}{2\pi i}\int_{(-1/2)}L(1+s,\psi)\Gamma(s)x^s\,ds. \label{eq:after_shifting}
\end{align}
The total error resulting from replacing $L(1,\psi_{t^2-4\ell n})$ by $S(x,\psi_{t^2-4\ell n})$ is bounded by 
\begin{align}
&\sum_{n\in I}\abs{\Phi(n)\log n}\cdot\abs{S(x,\psi_{t^2-4\ell n})-L(1,\psi_{t^2-4\ell n})} \\
&\ll_\varepsilon MB^\varepsilon x^{-1/2}\sum_{n\in I}\int_{\abs{\tau}\ll(B\ell)^\varepsilon}\abs{L(1/2+i\tau,\psi_{t^2-4\ell n})}\,d\tau \\
&\ll_\varepsilon MB^\varepsilon(B-A)^{3/4} x^{-1/2}\int_{\abs{\tau}\ll(B\ell)^\varepsilon}\bigg(\sum_{n\in I}\abs{L(1/2+i\tau,\psi_{t^2-4\ell n})}^4\bigg)^{1/4}\,d\tau, \label{eq:temp_bound}
\end{align}
where in the first step we applied \eqref{eq:after_shifting} followed by the rapid decay of $\Gamma(s)$, and in the second step we applied H\"older's inequality. Applying the (Lindel\"of-on-average) fourth-moment bound from \cite[Theorem 2]{HB} to the integrand, we can further bound \eqref{eq:temp_bound} as
\begin{align}
&\ll_{\varepsilon,\ell} MB^{3/4+\varepsilon} x^{-1/2}\max_{n\in I}\abs{t^2-4\ell n}^{1/4+\varepsilon} \\
&\ll_{\varepsilon,\ell} MB^{1+\varepsilon} x^{-1/2}.\label{eq:final_error}
\end{align}
    
Thus the sum over $L$-functions is given by
\begin{align*}
\sum_{n \in I} &\Phi(n)L(1,\psi_{t^2-4\ell n}) \log n \\
&= \sum_{n\in I}\Phi(n)S(x,\psi_{t^2-4\ell n})\log n + O_{\varepsilon,\ell}(MB^{1+\varepsilon}x^{-1/2}),\quad\text{ by \eqref{eq:final_error}}\\
&= \sum_{m\geq1}\frac{e^{-m/x}}{m}\sum_{n\in I}\Phi(n)\psi_{t^2-4\ell n}(m)\log n + O_{\varepsilon,\ell}(MB^{1+\varepsilon}x^{-1/2}) \\
&= \sum_{m\geq1} \frac{e^{-m/x}\widetilde{\psi}_t^{(\ell)}(m)}{m} \int_A^B \Phi(u) \,du \\
&\quad\,+ O\bigg((M+V) \sum_{m\geq1} \frac{e^{-m/x}\varphi(m^2)}{m}( \max_{u \in [A,B]} |E(u,m^2)| + \sqrt B)\bigg) + O_{\varepsilon,\ell} (MB^{1+\varepsilon} x^{-1/2}),\text{ by \eqref{eq:before_Lfunction}}.
\end{align*}
We turn our attention to the first error term, noting that $\frac{\varphi(m^2)}{m} = \varphi(m)$, which we bound above by $m$. 
The $\sqrt B$-term is bounded by
\[(M+V) \sum_{m\geq1}e^{-m/x} m \sqrt B \ll  (M+V) x^2\sqrt B.\]
By Theorem  \ref{thm:baker-thm-1}, as long as $x \le B^{1/4-\varepsilon}$, for any $C > 0$,
\begin{align*}
\sum_{m \geq1} e^{-m/x}m \max_{u \in [A,B]} E(u,m^2)& 
\ll \sum_{j\geq0}e^{-2^{j/2}/x} 2^{j/2} \sum_{2^j < m^2 \le 2^{j+1}}  \max_{u \in [A,B]} E(u,m^2) \\ 
&\ll_{C,\varepsilon} \sum_{j\geq0}e^{-2^{j/2}/x}2^{j/2}  2^{-{j/2}} B(\log B)^{-C-1} \\
&\ll_{C,\varepsilon} \frac{B}{(\log B)^C},
\end{align*}
where in the last line we used the fact that $\log x \ll \log B$. Returning to the sum over $L$-functions, we get that whenever $x \le B^{1/4-\varepsilon}$,
\begin{multline*}
\sum_{n \in I} \Phi(n)L(1,\psi_{t^2-4\ell n}) \log n 
= \sum_{m\geq1} \frac{e^{-m/x}\widetilde{\psi}_t^{(\ell)}(m)}{m} \int_A^B \Phi(u) \,du 
\\ + O((M+V)x^2 \sqrt B) + O_{C,\varepsilon}\left((M+V)\frac{B}{(\log B)^C}\right) 
+ O_{\varepsilon,\ell}(MB^{1 + \varepsilon} x^{-1/2}).
\end{multline*}
Choosing $x = \frac{M}{M+V}B^{3/14}$, we get that
\begin{align*}
\sum_{n \in I} &\Phi(n) L(1,\psi_{t^2-4\ell n}) \log n \\
& = \sum_{m\geq1} \frac{e^{-m/x}\widetilde{\psi}_t^{(\ell)}(m)}{m} \int_A^B \Phi(u) \,du + O(M^2(M+V)^{-1}B^{13/14}) + O_{C}\left((M+V)\frac{B}{(\log B)^C}\right) \\
& \quad\, + O_{\varepsilon,\ell}(M^{1/2}(M+V)^{1/2}B^{25/28 + \varepsilon}) \\ 
& = \sum_{m\geq1} \frac{e^{-m/x}\widetilde{\psi}_t^{(\ell)}(m)}{m} \int_A^B \Phi(u) \,du+ O_{C}\left((M+V)\frac{B}{(\log B)^C}\right).
\end{align*}
To finish the proof, it remains to estimate
\begin{align}
\sum_{m\geq1}\frac{(1-e^{-m/x})\abs{\widetilde{\psi}_t^{(\ell)}(m)}}{m} \leq 2\sum_{d\geq1}\frac{\mu^2(d)}{d^2}\sum_{\substack{r\geq1 \\ \text{squarefull}}}\frac{1-e^{-dr/x}}{r} \ll \sum_{d\geq1}\frac{\min\{1,\sqrt{d/x}\}}{d^2} \ll \frac{1}{\sqrt{x}}
\end{align}
as in the proof of \cite[Lemma 4.5]{bbld}.
\qed

\thispagestyle{empty}
{\footnotesize
\bibliographystyle{amsalpha}
\bibliography{reference}
}
\end{document}